\documentclass[11pt]{article}
\usepackage[utf8]{inputenc}
\usepackage{enumitem}
\usepackage[utf8]{inputenc}
\usepackage[english]{babel}
\usepackage[T1]{fontenc}
\usepackage{amsfonts}
\usepackage{amsmath}
\usepackage{amssymb}
\usepackage{amstext}
\usepackage{amsthm}
\usepackage{amscd}
\usepackage{cite}
\usepackage{mathtools}
\usepackage{setspace}
\usepackage{geometry}
\usepackage[hidelinks]{hyperref}
\usepackage{tikz-cd}
\usepackage{listings}
\usepackage{url}
\usepackage{comment}
\PassOptionsToPackage{hyphens}{url}\usepackage{hyperref}

\newtheorem{theorem}{Theorem}[section]
\newtheorem{corollary}[theorem]{Corollary}
\newtheorem{lemma}[theorem]{Lemma}

\theoremstyle{definition}
\newtheorem{definition}[theorem]{Definition}
\newtheorem{remark}[theorem]{Remark}
\newtheorem{example}[theorem]{Example}

\newcommand{\Q}{\mathbb{Q}}

\newcommand{\OO}{\mathcal{O}}
\newcommand{\Z}{\mathbb{Z}}
\newcommand{\F}{\mathbb{F}}

\newcommand{\A}{\mathcal{A}}

\newcommand{\p}{\mathfrak{p}}
\newcommand{\N}{\mathbb{N}}

\newcommand{\uu}{\underline}

\newcommand{\Rk}{\operatorname{Rank}}

\def\Gal{\operatorname{Gal}}

\def\lcm{\operatorname{lcm}}

\title{Divisibility sequences related to abelian varieties isogenous to a power of an elliptic curve}
\author{Stefan Barańczuk, Bartosz Naskręcki, Matteo Verzobio}
\date{}

\begin{document}
	
	\maketitle
	\begin{abstract}
		Let $A$ be an abelian variety defined over a number field $K$, $E/K$ be an elliptic curve, and $\phi:A\to E^m$ be an isogeny defined over $K$. Let $P\in A(K)$ be such that $\phi(P)=(Q_1,\dots, Q_m)$ with $\Rk_\Z(\langle Q_1,\dots, Q_m\rangle)=1$. We will study a divisibility sequence related to the point $P$ and show its relation with elliptic divisibility sequences.
	\end{abstract}
	\let\thefootnote\relax
	\section{Introduction}
	Let $A$ be an abelian variety defined over a number field $K$ and $P$ be a non-torsion point in $A(K)$. 
	Let $\OO_K$ denote the ring of integers of $K$ and $\A/\OO_K$ be the Néron model for $A/K$.
	
	Let $S$ be a finite set of primes in $K$. For each $n\geq 1$, define the integral ideal $C_n(\A,P,S)$ in $\mathcal{O}_K$ as 
	\[
	C_n(\A,P,S)=\prod_{\substack{\p: nP\equiv O\mod \p\\\p\notin S}} \p
	\]
	where with $nP\equiv O\mod \p$ we mean that $nP$ reduces to the identity in $\A$ reduced modulo $\p$.
	
	The goal of this paper is to find some examples of sequences $C_n(\A,P,S)$ such that, for all but finitely many $n$, $C_n(\A,P,S)$ has a primitive divisor, i.e. there exists a prime $\p$ that divides $C_n$ and does not divide $C_k$ for $k<n$. Notably, it is worth mentioning that the presence of a primitive divisor for almost all values of $n$ remains unaffected if we substitute the set $S$ with another finite set of primes.
	
	In the case when $A$ is an elliptic curve, $P$ is a point on it, and
	the curve $\A$ is defined via its Weierstrass equation, every term
	$C_n(A,P,S)$ is the product of the prime ideals not in $S$ which divide the
	$n$-term of the corresponding elliptic divisibility sequence
	$B_n(E,P)$, defined by the denominators of the $x$-coordinate of the
	point $nP$.
	For the properties of elliptic divisibility sequences, see \cite{phdthesis}. It is well-established that, for all but finitely many values of $n$, $B_n(E,P)$ contains a primitive divisor, as shown in \cite[Proposition 10]{silverman} and \cite[Main Theorem]{ChHa}. 
	
	Since the case when $A$ is an elliptic curve is well understood, we focus on higher-dimension cases.
	Let $A$ be an abelian variety of dimension $\geq 2$ defined over a number field $K$ and let $P\in A(K)$. Assume that $\Z P$ is Zariski dense in $A$. It is conjectured \cite[Proposition 9 and Conjecture 10]{silverman2005generalized} that there are infinitely many $n$ such that $nP$ does not reduce to the identity modulo any prime, outside a finite set of primes $S$. So, $C_n(\A,P,S)$ does not have a primitive divisor for infinitely many $n\in \N$. For a nice potential application of these sequences to Hilbert's tenth problem, see \cite[Remark (i), page 4]{cornelissen}.
	
	Since we are interested in sequences that have a primitive divisor for all but finitely many terms, we will focus on examples where $\Z P$ is not Zariski dense in $A$. More precisely, we will study the following case. Let $A$ be an abelian variety defined over a number field $K$ and let $P\in A(K)$. Assume that there is an elliptic curve $E$ and an isogeny $\phi$, both defined over $K$, such that $\phi:A\to E^m$ and $\phi(P)=(Q_1,\dots, Q_m)$. Assume that $\Rk_\Z(\langle Q_1,\dots, Q_m\rangle)=1$, i.e. the subgroup $\langle Q_1,\dots, Q_m\rangle$ of $E$ is isomorphic to $ \Z\oplus T$ where $T$ is a torsion group. Note that $\Z P$ is not Zariski dense in any subvariety of dimension at least $2$, so in particular it is not Zariski dense when $A$ has dimension at least $2$. We will prove the following.
	\begin{theorem}\label{thm:main}
		Let $A$ be an abelian variety defined over a number field $K$, let $\A/\OO_K$ be the Néron model for $A/K$, and let $P\in A(K)$. Let us assume the following:
		\begin{itemize}
			\item there is an elliptic curve $E$ and an isogeny $\phi$, both defined over $K$, such that $\phi:A\to E^m$ and $\phi(P)=(Q_1,\dots, Q_m)$;
			\item $\Rk_\Z(\langle Q_1,\dots, Q_m\rangle)=1$.
		\end{itemize}
		Then, there exists a finite set of primes $S$ in $K$, an integer $n_1\geq 1$, an elliptic curve $E_0$ defined over $K$, and $Q_0\in E_0(K)$ such that
		\[
		C_n(\A,P,S)=\begin{cases}
			1 \text{ if } n_1\nmid n, \\
			C_{n/n_1}(E_0,Q_0,S) \text{ if } n_1\mid n.
		\end{cases}
		\]
	\end{theorem}
    Let $Q_0'\in E_0(\overline{K})$ be such that $n_1Q_0'=Q_0$, and $K'$ be a Galois finite extension of $K$ such that $Q_0'\in E_0(K')$. Up to enlarging $S$ in Theorem \ref{thm:main}, we can assume that singular primes of $\A$ and $E_0$, and the prime divisors of the discriminant of $K'/K$ are in $S$. Let $S'$ be the set of primes in $K'$ that are over primes in $S$.
    Whence, we have
    		\[
		C_n(\A,P,S)\OO_{K'}=\begin{cases}
			1 \text{ if } n_1\nmid n, \\
			C_{n}(E_0,Q_0',S') \text{ if } n_1\mid n.
		\end{cases}
		\]
  So, one can see $C_n(\A,P,S)$ as an elliptic divisibility sequence with all entries with index not divisible by $n_1$ blanked out. The elliptic divisibility sequence $C_{n}(E_0,Q_0',S)$ is not defined over $K$ but over the finite extension $K'$.
    
	After the proof of this theorem, we will show how to compute the unique integer $n_1$ of Theorem \ref{thm:main} and we will show some examples.
	
	Regarding the problem of primitive divisors, we will prove the following.
	\begin{corollary}\label{cor:primdiv}
		Let $A$ be an abelian variety defined over a number field $K$, let $\A/\OO_K$ be the Néron model for $A/K$, let $P\in A(K)$, and let $S$ be a finite set of primes. Let us assume the following:
		\begin{itemize}
			\item there is an elliptic curve $E$ and an isogeny $\phi$, both defined over $K$, such that $\phi:A\to E^m$ and $\phi(P)=(Q_1,\dots, Q_m)$;
			\item $\Rk_\Z(\langle Q_1,\dots, Q_m\rangle)=1$.
		\end{itemize}
		Let $d_\phi$ be the degree of $\phi$ and let $\bar{\phi}:E^m\to A$ be such that $\phi\circ \bar{\phi}=d_\phi$.
		Then, $C_n(\A,P,S)$ has a primitive divisor for all but finitely many $n$ if and only if there exists $U'\in E(\overline{K})$ and $a_1,\dots, a_m\in\Z$ such that $\bar{\phi}(a_1U',\dots,a_mU')=P$.
	\end{corollary}
       \begin{remark}
        In the case when $A$ is defined over $\Q$, one can compute the terms of the sequence $C_n(\A,P,S)$ using a recursion formula. This can be done by combining the recurrence relation in \cite[Theorem 1.9]{recurrence} and Theorem \ref{thm:main}.
    \end{remark}
	\section*{Acknowledgements}
     We thank the anonymous referees for several useful comments that improved the exposition of this paper. In particular, we thank one of the referees for suggesting the alternative proof of Theorem \ref{thm:main} given in Remark \ref{rem:alternativeproof}. 
    
	The second author acknowledges the support by Dioscuri program initiated by the Max Planck Society, jointly managed with the National Science Centre (Poland), and mutually funded by the Polish Ministry of Science and Higher Education and the German Federal Ministry of Education and Research.
	
	The third author has received funding from the European Union’s Horizon 2020 research and 
	innovation program under the Marie Skłodowska-Curie Grant Agreement No. 
	101034413.
	
	We are also grateful to the University of Bristol for providing us with the access to Magma cluster CREAM.
	\section{Proof of the main theorem}\label{sec:proof}
	In the sequel, let $A$ be an abelian variety and $E$ be an elliptic curve defined over a number field $K$, and $\A/\OO_K$ be the Néron model for $A/K$.
	\begin{lemma}\label{lemma:Gd}
		Let $\phi':E^m\to A$, let $a_1,\dots, a_m\in \Z$ with $\gcd_i(a_i)=1$, and let $Z_1,\dots, Z_m\in E(\overline{K})$. For each $n\in \N$, define
		\[
		G_n=\{V\in E(\overline{K})\mid \phi'(a_1V+nZ_1,\dots ,a_mV+nZ_m)=O\}.
		\]
		There exists a unique non-negative integer $d$ such that for every integer $n$ the following
		equivalence holds: $G_n\neq \emptyset$ if and only if $d\mid n$. If $G_n= \emptyset$ for all $n\neq 0$, we have $d=0$. Moreover, for each $k\in\Z$, $\#G_{kd}=\#G_0$.
	\end{lemma}
	\begin{proof}
        If $G_n=\emptyset$ for all $n\geq 1$, the lemma is trivial. So, we assume that there exists $n\geq 1$ such that $G_n$ is not empty.
 
		Note that if $G_a$ and $G_b$ are not empty, then $G_{a+b}$ is not empty. Indeed, if $V_a\in G_a$ and $V_b\in G_b$, then $V_a+V_b\in G_{a+b}$. Let $d>0$ be the minimum positive integer such that $G_d$ is not empty. Then $G_n$ is not empty if and only if $n=kd$. If $G_n$ is not empty, then $\exists V_n\in G_n$ and there is a bijection between $G_0$ and $G_n$ given by $V\to V+V_n$.
	\end{proof}
	\begin{remark}
        The subset $G_0$ is always a subgroup of $E$ but in general the set $G_n$ does not form a group with the standard operations in $E$.
	\end{remark}
	\begin{definition}
		Let $\phi:A\to E^m$ be an isogeny. Let $d_\phi$ be the degree of the isogeny and $\bar{\phi}:E^m\to A$ be the unique isogeny such that $\phi\circ \bar{\phi}=d_\phi$. Such an isogeny exists due to \cite[Section II.7]{mumford1970abelian}. Hence, $\deg(\bar{\phi})=(d_\phi)^{2m-1}$ since $\deg(d_\phi)=(d_\phi)^{2m}$ by \cite[Theorem A.7.2.7]{hindry2013diophantine}.
	\end{definition}
	\begin{lemma}\label{lemma:bij}
		Let $\phi:A\to E^m$ and $\p$ be a prime in $K$ where $E$ has good reduction. Assume that $Z_1,\dots, Z_m\in E(\overline{K})$ are torsion points of order that divides $d_\phi$. Let $a_1,\dots, a_m\in \Z$ with $\gcd_i(a_i)=1$.
		If $\p$ does not divide $d_\phi$, there is a bijection between $G_n$, defined as 
        \begin{equation}\label{eq:Gnlem}
        \{V\in E(\overline{K})\mid \bar{\phi}(a_1V+nZ_1,\dots ,a_mV+nZ_m)=O\in E(\overline{K})\},
        \end{equation}
        and \[\{V\in E(\overline{\F_\p})\mid \bar{\phi}(a_1V+nZ_1,\dots ,a_mV+nZ_m)=O\in E(\overline{\F_\p})\}.\]
	\end{lemma}
	\begin{proof}
		Notice that $G_n\subseteq E(\overline{K})[d_\phi]$. Indeed, for $V\in G_n$,
		\begin{align*}
			(d_\phi a_1V,\dots ,d_\phi a_mV)&=d_\phi(a_1 V+nZ_1,\dots, a_m V+nZ_m)\\&=\phi\circ \bar{\phi}(a_1 V+nZ_1,\dots, a_m V+nZ_m)\\&=\phi(O)\\&=(O,\dots,O).
		\end{align*}
		Hence, $d_\phi V=O$ since $\gcd(a_i)=1$.
		We conclude by observing that there is a bijection between $E(\overline{K})[d_\phi]$ and the reduction modulo $\p$ of $E[d_\phi]$, cf. 
		\cite[Thm C.1.4]{hindry2013diophantine}, since $\p$ does not divide $d_\phi$.
	\end{proof}
    We are now ready to prove Theorem \ref{thm:main}. The techniques involved are related to those used in \cite{baranczuk2014remark}. In Remark \ref{rem:alternativeproof}, we show how to give a more geometric proof. Our argument here is more constructive, which will be useful later in Section \ref{sec:ex}. For readers primarily interested in the result and not in the computational details, we recommend consulting the alternative proof in Remark \ref{rem:alternativeproof}, which is shorter and easier to follow. We begin by sketching the overall strategy. The proof is divided into three parts:

    \begin{enumerate}[leftmargin=*]
        \item By assumption, $\phi(P)=(a_1R+T_1,\dots ,a_mR+T_m)$ for $R\in E(K)$ a non-torsion point and $T_i\in E(K)$ torsion points. We show that, outside finitely many primes, $nP\equiv O\mod \p$ implies $nT_i=0$ for all $1\leq i\leq m$. So, if $n$ is not a multiple of the order of $T_i$ for all $i$, then $C_n(\A,P,S)=1$. In the case when $n$ is a multiple of the order of $T_i$ for all $i$, it is easy to see that this case is equivalent to the case $T_1=T_2=\dots =T_m=O$ and so $\phi(P)=(a_1R,\dots ,a_mR)$.
        \item If $\phi(P)=(a_1R,\dots ,a_mR)$, we have $\overline{\phi}(a_1U'+Z_1,\dots ,a_mU'+Z_m)=P$ for $U'\in E(\overline{K})$ a non-torsion point and $Z_i\in E(\overline{K})$ torsion points of order dividing the degree of $\phi$. We show that, if $G_n$, as defined in \eqref{eq:Gnlem}, is empty, then $C_n(\A,P,S)=1$. This follows from the fact that $nP\equiv O\mod \p$ implies $nU'\in G_n$ modulo $\p$, and this cannot happen by Lemma \ref{lemma:bij}. So, we focus on the case when $G_n$ is non-empty and, after replacing $n$ with an appropriate multiple, we can assume $G_1$ is non-empty.
        \item Let $V_1\in G_1$ and notice $nV_1\in G_n$. By definition, $nP\equiv O\mod\p$ if and only if $nU'\in G_n$ modulo $\p$, or equivalently if and only if $n(U'-V_1)\in G_0$ modulo $\p$. So, if $E_0=E/G_0$ and $Q_0=U'-V_1\in E_0$, then $nP\equiv O\mod\p$ if and only if $nQ_0\equiv O\mod\p$ in $E_0(\F_\p)$. We conclude that $C_n(\A,P,S)=C_n(E_0,Q_0,S)$.
    \end{enumerate}
	\begin{proof}[Proof of Theorem \ref{thm:main}]
		Since the rank of the subgroup of $E(K)$ generated by the points $Q_{1}, \ldots, Q_{m}$ is equal to $1$, there exist a non-torsion point $R \in E(K)$, some integers $a_{1}, \ldots, a_{m}$ with $\gcd (a_{1}, \ldots, a_{m}) =1$, and torsion points $T_{i} \in E(K)$ such that, for every $i = 1, \ldots, m$,
		\begin{equation}\label{formula for Qi}
			Q_{i}=a_{i}R+T_{i}.
		\end{equation}
		Put 
		\begin{equation*}
			U=\sum_{i=1}^m b_{i} Q_{i}, 
		\end{equation*}
		where $b_{1}, \ldots, b_{m}$ are integers such that $\sum a_{i}b_{i} = \gcd (a_{1}, \ldots, a_{m})=1$. 
		Let us introduce the following notation for the sake of simplification of the exposition. Put
		\begin{equation*}
			\uu{Q}=[Q_{1}, \ldots, Q_{m}], \, \uu{a}=[a_{1}, \ldots, a_{m}],\,  \uu{T}=[T_{1}, \ldots, T_{m}], \,
			\uu{b}=\left[ 
			\begin{array}{c}
				b_{1}\\
				\vdots\\
				b_{m}
			\end{array}\right] .
		\end{equation*} 
		So,
		\begin{equation}\label{eq:defQ}
			\uu{Q}=\uu{a}R+\uu{T} 
		\end{equation}
		and
		\begin{equation}\label{eq:defU}
			U=\uu{Q}\cdot\uu{b}=\uu{a}\cdot\uu{b}R+\uu{T}\cdot\uu{b}=R+\uu{T}\cdot\uu{b}.  
		\end{equation}
		Let $\p$ be a prime of good reduction. If $nP\equiv O\mod \p$, then $nQ_i\equiv O\mod \p$ for each $i$ since $\phi(nP)=(nQ_1,\dots ,nQ_m)$. Hence, $nU\equiv O\mod \p$ because $U=\uu{Q}\cdot\uu{b}$. Moreover, by \eqref{eq:defQ} and \eqref{eq:defU}, \[
		\uu{Q}-\uu{a}U=\uu{a}R+\uu{T}-\uu{a}R-\uu{a}(\uu{T}\cdot\uu{b})=\uu{T}-\uu{a}(\uu{T}\cdot\uu{b}),
		\] 
        so $Q_j-a_jU$ is a torsion point for all $1\leq j\leq m$.
        Let $u_j>0$ be the order of $Q_j-a_jU$ and \begin{equation}
            \label{eq:u}
            u=\lcm_{1\leq j\leq m}\{u_j\}.
        \end{equation}
		
		Let $S$ contain the primes where $A$ and $E$ have bad reduction, and the primes that divide $d_\phi$ and $u$. Recall that a non-trivial $u$-torsion point is not the identity modulo a prime that does not divide $u$ (see \cite[Theorem VIII.7.1]{arithmetic}). Hence, if $n(Q_j-a_jU)\neq O$ and $n(Q_j-a_jU)\equiv O\mod \p$, then $\p$ must be in $S$ since $n(Q_j-a_jU)$ is a $u$-torsion point.
		
		If $nP\equiv O\mod \p$ for $\p\notin S$ and $u\nmid n$, then there exists $j$ such that $n(Q_j-a_jU)\neq O$. As we said, $n(Q_j-a_jU)\equiv O\mod \p$ and this contradicts the hypothesis $\p\notin S$. Therefore, if $u\nmid n$, then $C_n(\A,P,S)=1$. So, if $nP\equiv O\mod \p$ for $\p\notin S$, then $u\mid n$. From now on, we will focus on this case since we already know that $C_n(\A,P,S)=1$ if $u\nmid n$.
		
		Let $n$ be a multiple of $u$ and put $n_2=n/u$. Therefore, $nQ_j=na_jU+n(Q_j-a_jU)=na_jU$ since $n(Q_j-a_jU)=O$. Thus,
		\[
			n\uu{Q}=\uu{a}(nU).
		\]
		
		Let $\uu{Q}'=(Q_1',\dots ,Q_m')\in E^m(\overline{K})$ be such that $\bar{\phi}(\uu{Q}')=P$. Note that $d_\phi(n\uu{Q}')=\uu{a}(nU)$ since 
		\[
		d_\phi(n\uu{Q}')=\phi\left(\bar{\phi}(n\uu{Q}')\right)=\phi(nP)=n\uu{Q}=\uu{a}(nU).
		\]
		Let $U'\in E(\overline{K})$ be such that $d_\phi(U')=U$. Therefore, for each $j$, $u\uu{Q}_j'-a_juU'$ is a torsion point of order that divides $d_\phi$ since
		\[
		d_\phi(u\uu{Q}_j'-a_juU')=a_j(uU)-a_j(uU)=O.
		\] 
		Hence, there exists $\uu{Z}=[Z_1,\dots, Z_m]$ with $Z_i\in E(\overline{K})[d_\phi]$, such that
		\begin{equation}\label{eq:U'P}
			\bar{\phi}(\uu{a}(nU')+n_2\uu{Z})=\bar{\phi}(n\uu{Q}')=nP.
		\end{equation}
		
		Let 
		\begin{equation}\label{eq:Gn}
			G_n=\{V\in E(\overline{K})\mid \bar{\phi}(a_1V+nZ_1,\dots, a_mV+nZ_m)=O\}.  
		\end{equation}
		As we proved in Lemma \ref{lemma:bij}, given $\p\notin S$, $G_n$ is bijective to its reduction modulo $\p$. Let $d$ be the smallest positive integer such that $G_d$ is not empty. Note that $O\in G_{{d_\phi}}$ since $Z_i\in E(\overline{K})[d_\phi]$ and so, by Lemma \ref{lemma:Gd}, $d\mid d_\phi$.
		
		Assume $d\nmid n_2$ and $nP\equiv O\mod \p$ for $\p\notin S$. Then, \[\bar{\phi}(a(nU')+n_2\uu{Z})=nP\equiv O\mod \p\] and $nU'$ belongs to the reduction modulo $\p$ of $G_{n_2}$. So, $G_{n_2}$ modulo $\p$ is not empty and this contradicts the hypothesis that $d\nmid n_2$ since $G_{n_2}$ must be empty. Hence, if $d\nmid n_2$, $C_n(\A,P,S)=1$. Therefore, we focus on the case $d\mid n$. Put $n_3=n_2/d$ (and so $n=udn_3$).
		
		Since $G_d$ is not empty, let $V_d\in G_d$ and $W=(udU'-V_d)\in E(\overline{K})$. Note that $(n/d)V_d\in G_n$ since $V_d\in G_d$. Let $E_0=E/G_0$ and $Q_0\in E(\overline{K})$ be the image of $W$ under the isogeny $E\to E_0$. 
		Notice that there exists a finite field extension $K'$ of $K$ such that $E_0$ and $Q_0$ are defined over $K'$. Let $S'$ be the set of primes in $\mathcal{O}_{K'}$ that are over the primes in $S$.
		
		Let $\p'\notin S'$. We have $nP\equiv O\mod \p'$ if and only if \[\bar{\phi}(\uu{a}(nU')+n_2\uu{Z})=nP\equiv O\mod \p'\] by Equation \eqref{eq:U'P}. Moreover, $\bar{\phi}(\uu{a}(nU')+n_2\uu{Z})\equiv O\mod \p'$ if and only if $nU'\in G_{n_2}=(n_2/d)V_d+G_0=n_3V_d+G_0$ in the reduction modulo $\p'$. So, if $\p'\mid C_n(\A,P,S)\mathcal{O}_{K'}$, then $n_3W=n_3(udU'-V_d)=nU'-n_3V_d\in G_0$. Thus, $n_3Q_0$ reduces to the identity modulo $\p'$. Therefore, $\p'\mid C_{n_3}(E_0,Q_0,S')$.
		
		If $\p'\mid C_{n_3}(E_0,Q_0,S')$, then $nU'-n_3V_d=n_3W\in G_0$ modulo $\p'$. Therefore, $nU'\in G_{n_2}=(n_3)V_d+G_0$ modulo $\p'$ and then $nP=\bar{\phi}(n\uu{a}U'+n_2\uu{Z})=O\mod \p'$. Thus, $\p'\mid C_n(\A,P,S)\mathcal{O}_{K'}$. In conclusion,
		\[
		C_n(\A,P,S)\mathcal{O}_{K'}=C_{n_3}(E_0,Q_0,S')=C_{\frac{n}{ud}}(E_0,Q_0,S').
		\]
		
		To conclude the proof, we just need to show that $E_0$ and $Q_0$ are defined over $K$. Recall that $E_0$ and $Q_0$ are defined over $K'$. We can assume that $K'/K$ is a Galois extension. Note that $G_0$ is $\Gal(K'/K)$-invariant since, if $V\in G_0$, then
		\[
		\bar{\phi}(a_1V^\sigma,\dots, a_mV^\sigma)=\bar{\phi}(a_1V,\dots, a_mV)^\sigma=O
		\]
		for all $\sigma\in \Gal(K'/K)$.
		Therefore, by \cite[Remark III.4.13.2]{arithmetic}, $E_0=E/G_0$ is defined over $K$. Moreover,
		\[
		\bar{\phi}(\uu{a}W)=\bar{\phi}(\uu{a}udU'-\uu{a}V_d)=\bar{\phi}(\uu{a}udU'+d\uu{Z})=udP
		\]
		since $V_d\in G_d$ and by \eqref{eq:U'P}. Therefore, for all $\sigma\in \Gal(K'/K)$,
		\[
		\bar{\phi}(\uu{a}(W-W^\sigma))=\bar{\phi}(\uu{a}(W))-\bar{\phi}(\uu{a}(W))^\sigma=udP-udP^\sigma=O
		\]
		and then $W-W^\sigma\in G_0$. Since $Q_0$ is the image of $W$ under the isogeny $E\to E/G_0$, we have $Q_0^\sigma=Q_0$ for all $\sigma\in \Gal(K'/K)$. Then, $Q_0$ is defined over $K$.
	\end{proof}
	\begin{remark}
		Notice that the constant $n_1$, as defined in the statement of Theorem \ref{thm:main}, is equal to $ud$, where $u$ is defined in \eqref{eq:u} and $d$ is defined just after \eqref{eq:Gn}.
	\end{remark}
 \begin{remark}
           Observe that if $A$ is just $E^{m}$, and $P=(Q_1,\dots, Q_m)\in E^{m}(K)$, then we immediately get the following: there exists a finite set of primes $S$ in $K$, an integer $u\geq 1$, and $Q_0\in E(K)$ such that
\[
C_n(\A,P,S)=\begin{cases}
1 \text{ if } u\nmid n, \\
C_{n}(E,Q_0,S) \text{ if } u\mid n.
\end{cases}
\]
 \end{remark}
 \begin{remark}\label{rem:alternativeproof}
     As suggested by one of the anonymous referees, one can give an alternative, more geometric proof of Theorem \ref{thm:main}. However, in comparison to the proof written above we cannot directly extract the value $n_1$ from this alternative proof.
     
     Let $P$ be a fixed $K$-rational point on $A$. Let $V$ be the Zariski closure of the set $\Z P$. This is a $1$-dimensional algebraic subgroup of $A$, hence is smooth and is the union of disjoint irreducible components over $\overline{K}$. Let $E_0$ be the component containing $O = 0P$. As $E_0$ contains $O$, it is defined over $K$. Moreover, $V$ is the union of the varieties $kP + E_0$ where $k$ ranges over $\Z$. Let $n_1$ be the smallest positive integer such that $n_1P\in E_0$. We get that $V$ is the disjoint union of $E_0$, $P+E_0$, $2P+E_0$, ..., $(n_1-1)P + E_0$ and we let $Q_0 = n_1P \in E_0(K)$.

Let $R$ be the image of $P$ in the quotient abelian variety $A / E_0$. The point $R$ has finite order. Choose defining equations of $A, E_0\subset A$, and $A / E_0$ in such a way that the map $A \to A / E_0$ is just a restriction to the first so many coordinates and let $S$ contain all primes of bad reduction of these chosen sets of equations, as well as all primes $\p$ for which there is a $k \in \{1,2,...,n_1-1\}$ with $kR \equiv 0\mod \p$.

Then for all primes outside $S$ and all $n$ with $n_1 \nmid n$, we have \[C_n(A, P, S) \mid C_n(A/E_0, R, S) = 1.\]
Moreover, if $n_1 \mid n$, then $nP = (n/n_1)Q_0$, hence $C_n(A, P, S) = C_{n/n_1}(E_0, Q_0, S)$. 
 \end{remark}
	\begin{corollary}
		Let $A$ be an abelian variety defined over a number field $K$, let $\A/\OO_K$ be the Néron model for $A/K$, and let $P\in A(K)$ be a non-torsion point. Assume that there is an elliptic curve $E$ and an isogeny $\phi$, both defined over $K$, such that $\phi:A\to E^m$. Assume $\Rk_K(E)=1$. Then, there exists a finite set of primes $S$, an integer $n_1$, an elliptic curve $E_0$ defined over $K$, and $Q_0\in E_0(K)$ such that
		\[
		C_n(\A,P,S)=\begin{cases}
			1 \text{ if } n_1\nmid n, \\
			C_{n/n_1}(E_0,Q_0,S) \text{ if } n_1\mid n.
		\end{cases}
		\]
	\end{corollary}
	\begin{proof}
		We have $\phi(P)=(Q_1,\dots, Q_m)$ and $\Rk(\langle Q_1,\dots, Q_m\rangle)\leq \Rk_K(E)=1$. Since $P$ is a non-torsion point, we have $\Rk(\langle Q_1,\dots, Q_m \rangle)=1$. Hence, we apply Theorem \ref{thm:main}.
	\end{proof}
	\section{Some considerations and examples}\label{sec:ex}
	By the proof of Theorem \ref{thm:main}, $C_n(\A,P,S)$ is equal to $C_n(E,Q_0,S)$ if and only if $Q_j=a_jU$ for each $1\leq j\leq m$ (see \eqref{eq:u}) and $G_1$ is not empty (see Equation \eqref{eq:Gn}). In the next example, we show a case when this happens.
	\begin{example}
		The isogeny of this example is taken by \cite[Proposition 4]{examples}.
		Let $a_0=0$, $a_1=1$, and $a_2=-9$. Let $E$ be defined by $y^2=(x-a_1)(x-a_2)(x-a_3)$ and take $Q'=(9,-36)\in E(\Q)$. There is an isogeny $\phi':E^2\to J$ with $J$ the Jacobian of the hyperelliptic curve
		\[
		y^2 = 30233088x^6 +
		305690112x^4 + 305690112x^2 + 30233088,
		\]
		with $\deg\phi'=4$, and $\ker \phi'\subseteq E^2[2]$. So, there exists $\phi:J\to E^2$ with $\phi'\circ \phi=2$. Following the proof of \cite[Proposition 4]{examples},
		\[
		\phi'(Q',Q')=P=\left(x^2 + \frac{64}{7}x + 1, \frac{23639040}{49}x + \frac{414720}{7}\right)\in J(\Q)
		\]
		where we are using the Mumford representation for the points on the Jacobian of a hyperelliptic curve.
		Note that,
		\[
		\phi(P)=\phi(\phi'(Q',Q'))=(2Q',2Q')=(Q,Q)
		\] where $Q=2Q'=(25/16, -195/64 )\in E(\Q)$. 
		
		Since $\deg(\phi)=4$, there exists $\bar{\phi}$ such that $\bar{\phi}\circ \phi=4$ and we can take $\bar{\phi}=2\phi'$. Let $\overline{Q}\in E(\overline{\Q})$ be such that $Q'=2\overline{Q}$ and so $\bar{\phi}(\overline{Q},\overline{Q})=\phi'(Q',Q')=P$. Therefore, defining $G_n$ as in the proof of Theorem \ref{thm:main} (see in particular \eqref{eq:Gn}), we have $G_0=\{V\in E(\overline{\Q })\mid \bar{\phi}(V,V)=O\}$ and $G_1=G_0$. One can easily check that $G_0$ is the group of points $R\in E(\overline{\Q})$ such that $2R=(0,0)$ or $2R=O$. Using MAGMA \cite{Magma}, we can compute that $E_0=E/G_0$ is the elliptic curve $y^2 = x^3 + 8x^2 + 36x + 288$ and $Q_0=(8,-40)\in E_0(\Q)$, where $Q_0$ is the image of $\overline{Q}$ under the isogeny $E\to E_0$. Hence, replicating the work in the proof of Theorem \ref{thm:main}, we have $C_n(J,P,S)=C_n(E_0,Q_0,S)$, where $S=\{2\}$. Using MAGMA,
		we compute the first terms of the two sequences.
		\begin{center}
			\begin{tabular}{|c|c| c|}
				\hline
				n&$C_n(J,P,S)$&$C_n(E,Q_0,S)$\\
				\hline
				$1$&$1$&$1$ \\
				\hline
				2&$1$&$1$  \\
				\hline
				3&$7\cdot17\cdot41 $&$7\cdot17\cdot41 $ \\
				\hline
				4&$13\cdot29\cdot101$&$13\cdot29\cdot101$  \\
				\hline
				5&$103\cdot113\cdot1087\cdot2377$ &$103\cdot113\cdot1087\cdot2377$ \\
				\hline
				6&$7\cdot11\cdot17\cdot41\cdot89\cdot2713\cdot8329 $&$7\cdot11\cdot17\cdot41\cdot89\cdot2713\cdot8329 $\\
				\hline
				7&$23\cdot23497\cdot156671\cdot48883577521$&$23\cdot23497\cdot156671\cdot48883577521$\\
				\hline
			\end{tabular}
		\end{center}
		In particular, by \cite[Proposition 10]{silverman}, $C_n(J,P,S)$ has a primitive divisor for all but finitely many terms. Note that this agrees with Corollary \ref{cor:primdiv}.
	\end{example}
	\begin{remark}
		Notice that the abelian variety of the previous example is not isomorphic to the square of an elliptic curve. Indeed, there is no genus $2$ curves on $E^2$. To prove this, it is sufficient to check that $E$ does not have complex multiplication, as shown in \cite[Theorem 1]{kani2014jacobians}. The endomorphism ring of $E$ is $\Z$ and it is computed in the LMFDB database \cite{lmfdb}.
	\end{remark}
	One may wonder if $G_1$ can be empty. In the next remark, we show that this can happen.
	\begin{remark}
		Let $K$ be a number field, and let $E/K$ be an elliptic curve with $E(K)[2]=\langle T_1,T_2\rangle$ for $T_1,T_2\in E(K)$ being two different points of order $2$ and $\Rk_K(E)\geq 1$. Let $H\subseteq E^2$ be the subgroup generated by $(T_1,T_1)$, $(T_2,T_2)$, and $(T_1,T_2)$. Let $A=E^2/H$ and $\bar{\phi}:E^2\to A$ be the isogeny with kernel $H$. So, $\bar{\phi}$ has degree $8$ and $\ker \bar{\phi}\subseteq E^2[2]$. By the properties of the quotient, there exists a map $\phi:A\to E^2$ such that $\phi\circ\bar{\phi}=[2]$. Let $U'\in E(K)$ be a non-torsion point and $P=\bar{\phi}(U'+T_1,U')$. So, \[\phi(P)=\phi(\bar{\phi}(U'+T_1,U'))=2(U'+T_1,U')=(2U',2U').\] Hence, we are in the hypothesis of Theorem \ref{thm:main}. Following the definition of $G_n$ in \eqref{eq:Gn},
		\[
		G_1=\{V\in E(\overline{K})\mid \bar{\phi}(V+T_1,V)=O\}.
		\]
		As we proved in Lemma \ref{lemma:bij}, we have $G_1\subseteq E(\overline{K})[2]$. By definition, $\bar{\phi}(V,V)=O$ for each $V\in E(\overline{K})[2]$. Hence, $\bar{\phi}(V+T_1,V)=\bar{\phi}(T_1,O)\neq O$ since $(T_1,O)\notin H$. So, $G_1$ is empty.
	\end{remark}
	\begin{example}\label{ex}
		Let $A$ and $P$ be as in the previous remark. Let $S$ be the set of primes over $2$ and where $A$ has bad reduction. We compute $C_n(\A,P,S)$.
		
		Let $n$ be odd and $\p\notin S$. Then $\p\mid C_n(\A,P,S)$ if and only if $(nU'+nT_1,nU')=(nU'+T_1,nU')$ reduces to a point in $H$ modulo $\p$ since $nP=\bar{\phi}(nU'+nT_1,nU')$ and $\ker(\bar{\phi})=H$. Notice that every point $(R_1,R_2)\in H$ is such that $R_1-R_2$ is equal to $O$ or $T_1-T_2$. Since $nU'+T_1-nU'=T_1$, we have that $(nU'+T_1,nU')$ does not reduce to a point in $H$ modulo $\p$. So, $C_n(\A,P,S)=1$.
		
		Let $n$ be even and $\p\notin S$. So, $\p\mid C_n(\A,P,S)$ if and only if $(nU'+nT_1,nU')=(nU',nU')$ reduces to a point in $H$ modulo $\p$. Hence, $\p\mid C_n(\A,P,S)$ if and only if $nU'$ is a $2$-torsion point modulo $\p$ and then if and only if $2nU'$ reduces to the identity modulo $\p$. So, $C_n(\A,P,S)=C_n(E,2U',S)$. In conclusion,
		\[
		C_n(\A,P,S)=\begin{cases}
			1 \text{ if } 2\nmid n, \\
			C_n(E,2U',S) \text{ if } 2\mid n.
		\end{cases}
		\]
	\end{example}
    \begin{example}
    We make Example \ref{ex} explicit in one case. Let $E$ be defined by $y^2=x^3-20x-21$, let $U'=(-3,4)\in E(\Q)$ be a non-torsion point, and let $T_1=(-1,0)$ and $T_2=(5,0)$ be two $2$-torsion points in $E$. This is the curve 288.b3 in the LMFDB \cite{lmfdb}. We have that $E(\Q)$ is generated by $U'$, $T_1$, and $T_2$. Let $H\subseteq E^2$ be the subgroup generated by $(T_1,T_1)$, $(T_2,T_2)$, and $(T_1,T_2)$. Let $A=E^2/H$ and let $T=(T_1,O)\in A(\Q)$, that is the only rational non-trivial $2$-torsion point in $A$. Let $P=(U',U')+T\in A(\Q)$ and $S=\{2,3\}$. Following Example \ref{ex},
    \[
		C_n(\A,P,S)=\begin{cases}
			1 \text{ if } 2\nmid n, \\
			C_n(E,2U',S) \text{ if } 2\mid n.
		\end{cases}
		\]
        We show that the equality holds also for $S=\{\emptyset\}$.
    Let $p$ be equal to $2$ or $3$, and notice $T_1\equiv T_2 \mod p$ is a non-singular point. If $nP\equiv O\mod p$, then $(nU'+nT_1,nU')\in H$ modulo $p$ and so $nT_1=nU'+nT_1-nU'$ must be equal to the identity modulo $p$. This happens if and only if $n$ is even. If $n$ is even, then $nP=(nU',nU')$. If $p=2$, then $nU'\equiv O\mod p$ for all $n$ even and so $nP\equiv O\mod p$. If $p=3$, then $nU'\notin H$ modulo $p$ for $n\equiv 1,2\mod 3$ and $nU'\equiv O\mod p$ for $n\equiv 0\mod p$. We conclude that
      \[
		C_n(\A,P,\{\emptyset\})=\begin{cases}
			1 \text{ if } 2\nmid n, \\
			C_n(E,2U',\{\emptyset\}) \text{ if } 2\mid n.
		\end{cases}
		\]
    We compute the first few terms.
    \begin{center}
	\begin{tabular}{|c|c| c|}
				\hline
				$n$&$C_n(\A,P,\{\emptyset\})$&$C_n(E,2U',\{\emptyset\})$\\
				\hline
				$1$&$1$&$2$ \\
				\hline
				$2$&$2\cdot 5\cdot 11\cdot 13$&$2\cdot 5\cdot 11\cdot 13$ \\
				\hline
				$3$&$1$&$ 2\cdot 3 \cdot7 \cdot17\cdot 19\cdot 23\cdot 263$ \\
				\hline
				$4$&$2 \cdot5 \cdot11\cdot 13 \cdot67 \cdot197 \cdot19249 \cdot21649$&$2 \cdot5 \cdot11\cdot 13 \cdot67 \cdot197 \cdot19249 \cdot21649$  \\
				\hline
				$5$&$1$ &$2\cdot 37\cdot 43\cdot 73 \cdot937\cdot 1583\cdot$ \\
                & &$\cdot 1867\cdot 2089\cdot 3041\cdot 21601$ \\
				\hline
				$6$&$2 \cdot3 \cdot5 \cdot7 \cdot11\cdot 13\cdot 17\cdot 19\cdot 23\cdot 191\cdot 251\cdot$ &
                $2 \cdot3 \cdot5 \cdot7 \cdot11\cdot 13\cdot 17\cdot 19\cdot 23\cdot 191\cdot 251\cdot$ \\
                &$\cdot263\cdot 311\cdot 16103\cdot 1786451\cdot 385044001$ &$\cdot263\cdot 311\cdot 16103\cdot 1786451\cdot 385044001$\\
				\hline
			\end{tabular}
		\end{center}
    \end{example}
	Now, we prove Corollary \ref{cor:primdiv}. To do that, we need a preliminary lemma. We will use the same notation of Section \ref{sec:proof}.
	\begin{lemma}\label{lemma:G1}
		Let $G_n$ be as in \eqref{eq:Gn}. Let $\alpha:E^m\to E^m$ be the isogeny that sends \[(P_1,\dots, P_m)\to (a_1P_1,\dots, a_mP_m).\] Then, $G_n$ is empty if and only if the following hold:
		\begin{itemize}
			\item $\alpha(\Delta[d_\phi])\subseteq\ker(\bar{\phi})$;
			\item $\bar{\phi}(nZ_1,\dots, nZ_m)\neq O$.
		\end{itemize}
		With $\Delta[d_\phi]$ we mean $\{(V,\dots, V)\in E^m(\overline{K})\mid V\in E(\overline{K})[d_\phi]\}$.
	\end{lemma}
	\begin{proof}
		Recall that $G_n\subseteq E(\overline{K})[d_\phi]$, that is a $2$-dimensional module over $\Z/d_\phi\Z$. Note that $\ker(\bar{\phi})$ is a $(2m-1)$-dimensional sub-module of $E^m(\overline{K})[d_\phi]$ and that $\alpha(\Delta[d_\phi])$ is a $2$-dimensional sub-module. Moreover, the image of $\bar{\phi}(E^m(\overline{K})[d_\phi])$ has dimension $1$.
		
		Assume that $G_n$ is empty. If $\bar{\phi}(nZ_1,\dots, nZ_m)= O$, then $O\in G_n$, contradiction. So, $\bar{\phi}(nZ_1,\dots, nZ_m)\neq O$. 
		If $\alpha(\Delta[d_\phi])\not\subseteq\ker(\bar{\phi})$, then there is $V\in E[d_\phi]$ such that $\bar{\phi}(a_1V,\dots, a_mV)\neq O$. 
		The image of $\bar{\phi}(E^m(\overline{K})[d_\phi])$ has dimension $1$ and so there is $i\leq d_\phi$ such that \[-i\bar{\phi}(a_1V,\dots, a_mV)=\bar{\phi}(nZ_1,\dots, nZ_m).\] So, $iV\in G_n$ and $G_n$ is not empty, contradiction. 
		Therefore, $\alpha(\Delta[d_\phi])\subseteq\ker(\bar{\phi})$.
		
		Now, we prove the only if statement. We have
		\begin{align*}
			\bar{\phi}(a_1V+nZ_1,\dots, a_mV+nZ_m)&=\bar{\phi}(\alpha(V,\dots ,V))+\bar{\phi}(nZ_1,\dots, nZ_m)\\&=\bar{\phi}(nZ_1,\dots, nZ_m)\\&\neq O
		\end{align*}
		for each $V\in E(\overline{K})[d_\phi]$ and then $G_n$ is empty.
	\end{proof}
	\begin{proof}[Proof of Corollary \ref{cor:primdiv}]
		Notice that, if $S$ and $S'$ are two finite sets of primes in $K$, then, by definition, $C_n(\A,P,S)$ has a primitive divisor for all but finitely many $n$ if and only if $C_n(\A,P,S')$ does. So, we have to prove the corollary only for the set of primes $S$ as in Theorem \ref{thm:main}. 
		
		By \cite[Proposition 10]{silverman} and Theorem \ref{thm:main}, $C_n(\A,P,S)$ has a primitive divisor for all but finitely many $n$ if and only if $n_1=1$, where $n_1$ is defined in the statement of Theorem \ref{thm:main}. 
		
		As we noted at the beginning of Section \ref{sec:ex}, this happens if and only if $Q_j-a_jU=O$ for each $j$ and $G_1$ is not empty (see Equation \eqref{eq:Gn}).
		
		Assume that $C_n(\A,P,S)$ has a primitive divisor for all but finitely many $n$. Therefore, $\phi(P)=(a_1U,\dots,a_mU)$. Moreover, $G_1$ is non-empty and then, by Lemma \ref{lemma:G1}, $\alpha(\Delta[d_\phi])\not\subseteq\ker(\bar{\phi})$ or $\bar{\phi}(Z_1,\dots, Z_m)= O$. In the second case, 
		\[
		\bar{\phi}(a_1U',\dots, a_mU')=\bar{\phi}(a_1U'+Z_1,\dots, a_mU'+Z_m)= P
		\]
		by \eqref{eq:U'P}. Assume that $\bar{\phi}(Z_1,\dots, Z_m)\neq  O$ and then $\alpha(\Delta[d_\phi])\not\subseteq\ker(\bar{\phi})$. Since $\bar{\phi}(E^m[d_\phi])$ has dimension $1$ as $\Z/d_\phi\Z$-module, there exists $V\in E(\overline{K})[d_\phi]$ such that \[\bar{\phi}(a_1V,\dots, a_m V)=\bar{\phi}(Z_1,\dots, Z_m).\] So,
        \[
        \bar{\phi}(a_1(U'+V),\dots, a_m(U'+V))=\bar{\phi}(a_1U'+Z_1,\dots, a_mU'+Z_m)= P.
        \]
        In both cases, we can find a point $U'$ (or $U'+V)$ such that
        \[
        \bar{\phi}(a_1U',\dots, a_mU')=P
        \]
        and we are done.
		
		Now, we prove the only if statement. By hypotheses, $Z_j=O$ for each $j$ and then $G_1$ is not empty (by Lemma \ref{lemma:G1}). So, $C_n(\A,P,S)$ is equal to an elliptic divisibility sequence and then $C_n(\A,P,S)$ has a primitive divisor for all but finitely many $n$.
	\end{proof}
	\normalsize
	\baselineskip=17pt
	\bibliographystyle{plain}
	\bibliography{bibliohy}
	Stefan Bara\'{n}czuk, Faculty of Mathematics and Computer Science, Adam Mickiewicz University in Poznań, ul. Uniwersytetu Poznańskiego 4,
	61-614, Poznań, Poland\\
	\textit{E-mail address}: \url{stefbar@amu.edu.pl}\\
	Bartosz Naskręcki, Faculty of Mathematics and Computer Science, Adam Mickiewicz University in Poznań, ul. Uniwersytetu Poznańskiego 4,
	61-614, Poznań, Poland;\\
	Institute of Mathematics, Polish Academy of Sciences, ul. Jana i Jędrzeja Śniadeckich 7, 00-654 Warszawa, Poland\\
	\textit{E-mail address}: \url{bartosz.naskrecki@amu.edu.pl}\\
	Matteo Verzobio, IST Austria, Am Campus 1, Klosterneuburg, Austria\\
	\textit{E-mail address}: \url{matteo.verzobio@gmail.com}\\
\end{document}